\documentclass[3p]{elsarticle}
\usepackage{amssymb,amsfonts,amsmath,bm}
\usepackage{graphicx,multicol,algorithm,epstopdf}

\usepackage{wrapfig,caption,color,ulem,url,float}
\usepackage{microtype,epsf}
\usepackage[colorlinks=true,breaklinks=true,pdftex]{hyperref}

\newtheorem{thm}{Theorem}
\newdefinition{definition}{Definition}
\newdefinition{rmk}{Remark}
\newdefinition{cor}{Corollary}
\newdefinition{proposition}{Proposition}
\newtheorem{lem}{Lemma}

\newdefinition{example}{Example}

\newcommand{\R}{\mathbb{R}}
\newcommand{\be}[1]{\begin{equation}\label{#1}}
\newcommand{\ee}{\end{equation}}

\journal{arXiv}
\begin{document}
\begin{frontmatter}
\title{Interpolation of scattered data in $\mathbb{R}^3$ using minimum $L_p$-norm networks, $1<p<\infty$}
\author[rvt]{K.~Vlachkova}
 \ead{krassivl@fmi.uni-sofia.bg}
 \address[rvt]{Faculty of Mathematics and Informatics, Sofia University ``St. Kliment Ohridski"\\
Blvd. James Bourchier 5, 1164 Sofia,  Bulgaria}

\date{\today}
\begin{abstract}
We consider the extremal problem of interpolation of scattered data in $\mathbb{R}^3$ by smooth curve networks with minimal
$L_p$-norm of the second derivative for $1<p<\infty$. The problem for $p=2$ was set and solved by
Nielson (1983) \cite{N}. Andersson et al. (1995) \cite{AEIV} gave a new proof of Nielson's result by using a different approach. Partial results for the problem for $1<p<\infty$ were announced without proof
in (Vlachkova (1992) \cite{V2}).
Here we present a complete characterization of the solution for
$1<p<\infty$. Numerical experiments are visualized and presented to illustrate and support our results.
\end{abstract}

\end{frontmatter}
\section{Introduction}

Scattered data interpolation is a fundamental problem in approximation
theory and CAGD. It finds applications in a variety of fields such as automotive, aircraft and ship design, architecture, medicine, computer graphics, and more. Recently, the problem has become particularly relevant in bioinformatics and scientific visualization. The interpolation of scattered data in $\mathbb{R}^3$ attracted a considerable amount of research. Different methods and approaches were proposed and discussed, excellent surveys are, e.\,g.~\citep{FN,LF,MLLMPDS}, see also~\citep{D,Am,ALP,CG}.

Consider the following problem: Given {\it scattered data}
$(x_i,y_i,z_i)$ $\in \mathbb{R}^3$,\ $i=1,\dots ,n$, that is points $V_i=(x_i,y_i)$ are different and non-collinear, find a bivariate function
$F(x,y)$ defined in a certain domain $D$ containing points $V_i$, such that $F$ possesses continuous partial derivatives up to a given order
and $F(x_i,y_i)=z_i$.

\citet{N} proposed a three steps method for solving the problem as follows:

\smallskip
{\sl Step 1.} {\it Triangulation}. Construct a triangulation $T$ of
$V_i,\ i=1,\dots n$.

\smallskip
{\sl Step 2.} {\it Minimum norm network}. The interpolant $F$ and
its first order partial derivatives are defined on the edges of $T$
to satisfy an extremal property. The obtained minimum norm network
is a cubic curve network, i.~e. on every edge of $T$ it is a cubic
polynomial.

\smallskip
{\sl Step 3.} {\it Interpolation surface}. The obtained network is extended to $F$ by an
appropriate {\it blending method}.

\citet{AEIV} paid special attention to Step~2 of the
above method, namely the construction of the minimum norm network. Using a different approach,
the authors gave a new proof of Nielson's result. They constructed a system
of simple linear curve networks called {\it basic curve networks} and then
represented the second derivative of the minimum norm network as a linear combination of these
basic curve networks.
\smallskip

The problem of interpolation of scattered data by minimum $L_p$-norms networks for  $1<p<\infty$ was considered in~\citep{V2}
 where sufficient conditions for the solution were formulated without proof.

In this paper we prove the existence and the uniqueness of the solution to the problem for $1<p<\infty$ and provide its complete characterization using the basic curve networks defined in~\citep{AEIV}.

 The paper is organized as follows. In Sect.~2 we introduce notation, formulate the extremal problem for interpolation by minimum $L_p$-norms networks for  $1<p<\infty$, and present some related results. In Sect.~3 we prove the existence and the uniqueness of the solution to the problem for $1<p<\infty$. In Sect.~4 we establish a full characterization of the solution. In final Sect.~4 we present the results from our experimental work.
Based on
numerical solving of nonlinear systems of equations we apply computer modeling and visualization
tools to illustrate and support our results.

\section{Preliminaries and related results}

Let $n\geq 3$ be an integer and $P_i:=(x_i,y_i,z_i),\ i=1,\dots ,n$
be different points in $\mathbb{R}^3$. We call this set of points {\it
data}. The data are {\it scattered} if the projections
$V_i:=(x_i,y_i)$ onto the plane $Oxy$ are different and
non-collinear.
\begin{definition}\label{def1}
A collection of non-overlapping, non-degenerate triangles in
$\mathbb{R}^3$ is a {\it triangulation} of the points $V_i,\ i=1,\dots ,n$,
if the set of the vertices of the triangles coincides with the set
of the points $V_i,\ i=1,\dots ,n$.
\end{definition}
Hereafter we assume that a triangulation $T$
of the points $V_i,\ i=1,\dots ,n$, is given and fixed. The union of all triangles in $T$ is
a polygonal domain which we denote by $D$.
In general $D$ is a collection of polygons with holes. The set of
the edges of the triangles in $T$ is denoted by $E$. If there is an
edge between $V_i$ and $V_j$ in $E$, it will be referred to by
$e_{ij}$ or simply by $e$ if no ambiguity arises.

\begin{definition}
{A {\it curve network} is a collection of real-valued univariate functions
$\{f_e\}_{e\in E}$
defined on the edges in $E$.}
\end{definition}
With any real-valued bivariate function $F$ defined on $D$ we naturally
associate the curve network defined as the restriction of $F$ on the edges
in $E$, i.~e. for $e=e_{ij}\in E$,
\begin{equation} \label{e0}
\begin{split}
  f_e(t) := F\Bigl(\bigl(1-\frac{t}{\| e\|}\bigr)x_i+\frac{t}{\| e\|}\,x_j,\,
                   \bigl(1-\frac{t}{\| e\|}\bigr)y_i+\frac{t}{\| e\|}\,y_j\Bigr), \\
  \text{where } 0\le t\le \| e\| \ \text{and} \ \|e\|=\sqrt{(x_i-x_j)^2+(y_i-y_j)^2}.
\end{split}
\end{equation}
Furthermore, according to the context $F$ will denote either a
real-valued bivariate function or a curve network defined by
(\ref{e0}). For $p$, such that $1<p<\infty$, we introduce the
following class of {\it smooth interpolants}
$$
{\cal F}_p:=\{ F(x,y)\, | \, F(x_i,y_i)=z_i,\ i=1,\dots ,n,\
\partial F/\partial x,
\partial F/\partial y \in C(D),\ f'_e\in AC,\ f^{\prime\prime}_e\ \in
L_p,\ e\in E\},
$$
where $C(D)$ is the class of bivariate continuous functions defined in $D$,
$AC$ is the class of univariate absolutely continuous functions defined in $[ 0,\| e\|]$, and $L_p$ is the class of univariate functions defined in $[ 0,\| e\|]$ whose p-th power of the absolute value is Lebesgue integrable.
The restrictions on $E$ of the functions in ${\cal F}_p$
form the corresponding class of so-called {\it smooth interpolation
curve networks}
\be{e31}
{\cal C}_p(E):=\left\{ F_{|E}=\{f_e\}_{e\in E}\ |\ F (x,y)\in {\cal
F}_p,\ e\in E\right\}.
\ee
We note that the class ${\cal F}_p$ is nonempty
since, e.~g. Clough-Tocher
~\cite{CT}
and Powell-Sabin~\cite{PS}
 interpolants belong to it. Hence ${\cal
C}_p(E)$ is nonempty too. The smoothness of the interpolation curve network $F=\{f_e\}_{e\in E}\in {\cal C}_p(E)$
geometrically means that if we consider the graphs of functions $\left\{f_e\right\}_{e\in E}$ as curves in $\R^3$ then at every point $V_i,\ i=1,\dots ,n$, they have a common tangent plane.

Inner product and $L_p$-norm are defined in ${\cal C}_p(E)$ by
 \begin{eqnarray*}
&&\langle F, G\rangle=\int_E FG=\sum_{e\in E}\int
_0^{\|e\| }f_e(t)g_e(t)dt,\\
&&\|F\|_p:= \left( \sum_{e\in E} \int_0^{\| e\|
}|f_e(t)|^pdt\right)^{1/p}, \quad 1\leq p<\infty ,
\end{eqnarray*}
where $F\in {\cal C}_p(E)$ and
$G:=\{g_e\}_{e\in E}\in {\cal C}_p(E)$.
We denote the networks of the
second derivative of $F$ by $F^{\prime\prime}:=\{f''_e\}_{e\in E}$ and
consider the following extremal problem:
$$({\bf P}_p)\quad  \mbox{\it\ Find\ }F^*\in {\cal C}_p(E)\  \mbox{\it\
such that }\| F^{*\prime\prime}\|_p=\inf_{F\in {\cal C}_p(E)}\|
F''\|_p.
$$

{Problem $({P}_p)$ is a generalization of the classical univariate extremal problem $(\tilde{P}_p)$ for interpolation of data in $\R^2$ by a univariate function with minimal $L_p$-norm of the second
derivative. The latter was studied by Holladay~\cite{Holl} for $p=2$ and by de Boor~\cite{deB1} in more general settings.\footnote{C. de Boor studied the more general problem of minimum $L_p$-norm of the $k$-th derivative, $k\geq 1$, $1<p\leq\infty$.}
Holladay~\cite{Holl} proved that the natural interpolating cubic spline is the unique solution to $(\tilde{P}_2)$.
Nielson's approach to construct minimum norm network can be seen as an extention of Holladay's proof~\citep{Holl}.}

For $i=1,\dots ,n$ let $m_i$ denote the degree of the vertex $V_i$,
i.~e. the number of the edges in $E$ incident to $V_i$.
Furthermore, let $\{ e_{ii_1},\dots ,e_{ii_{m_i}}\}$ be the edges
incident to $V_i$ listed in clockwise order around $V_i$. The first
edge $e_{ii_1}$ is chosen so that the coefficient
$\lambda_{1,i}^{(s)}$ defined below is not zero - this is always
possible. A {\it basic curve network} $B_{is}$ is defined on $E$ for
any pair of indices $is$, such that $i=1,\dots ,n$ and $s=1,\dots
,m_i-2$, as follows (see Fig. \ref{basic}):
\be{e5}
B_{is}:= \left\{
\begin{array}{ll}
\lambda^{(s)}_{r,i}\left( 1-\frac{t}{\| e_{ii_{s+r-1}} \|
}\right) & {\rm on}\ e_{ii_{s+r-1}},\ r=1,2,3,\\
& \ \ \ 0\leq t\leq \| e_{ii_{s+r-1}}\| \\[1ex]
0 & {\rm on\ the\ other\ edges\ of\ }E.
\end{array}
\right.
\ee

The coefficients $\lambda_{r,i}^{(s)},\ r=1,2,3$, are
uniquely determined to sum to one and to form a zero linear
combination of the three unit vectors along the edges
$e_{ii_{s+r-1}}$ starting at $V_i$.
\begin{figure}[ht]
\begin{center}
\begin{minipage}{4.3cm}
\begin{center}
{\bfseries \includegraphics[width=1.\textwidth]{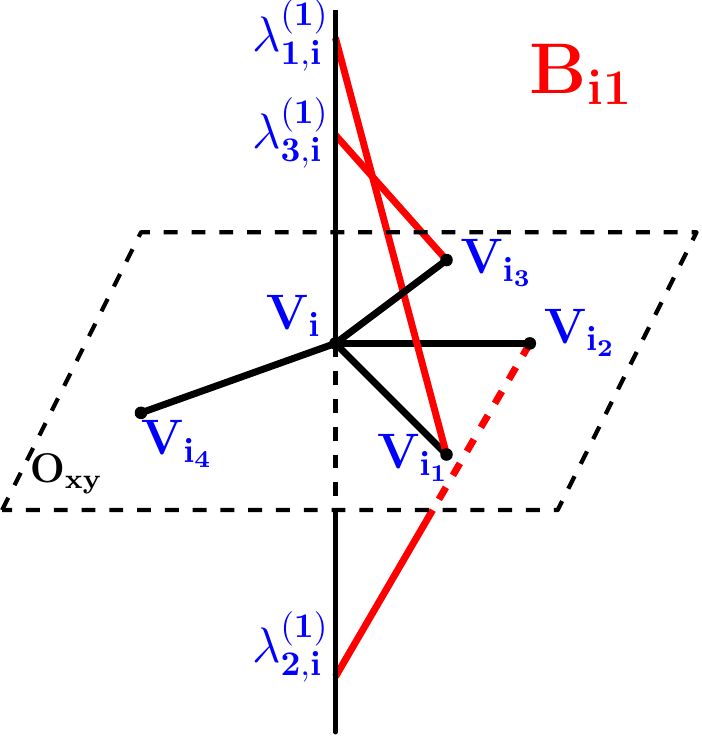}}
\end{center}
\end{minipage}
\hspace*{1.5cm}
\begin{minipage}{4.3cm}
\begin{center}
{\bfseries \includegraphics[width=1.\textwidth]{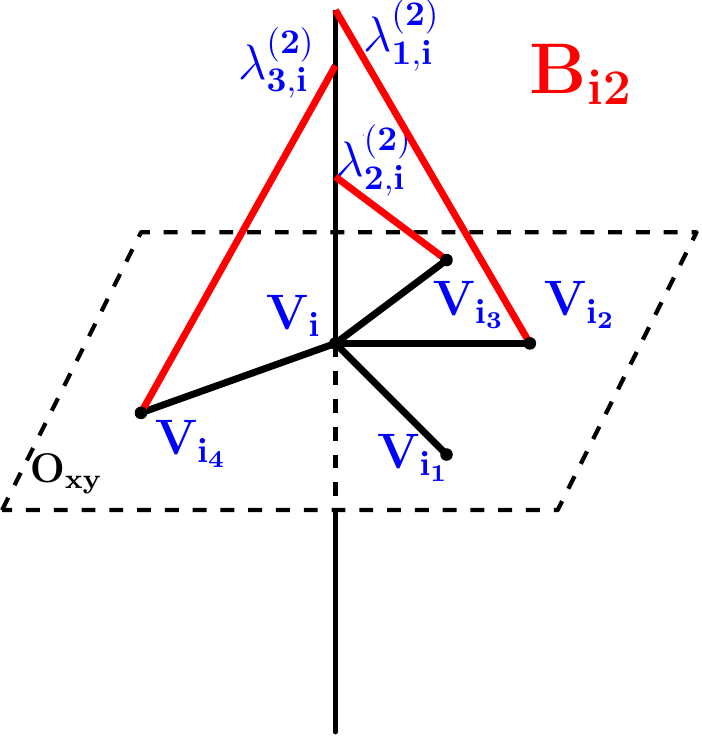}}
\end{center}
\end{minipage}
\caption{The basic curve networks for vertex $V_i$ of degree $m_i=4$}\label{basic}
\end{center}
\end{figure}

Note that basic curve networks are associated with points that have
at least three edges incident to them.
We denote by
$N_B$ the set of pairs of indices $is$ for which a basic curve
network is defined, i.~e.,
$$ N_B:=\{is\ |\ m_i\geq 3,\ i=1,\dots, n,\ s=1,\dots, m_i-2\}.$$

With each basic curve network $B_{is}$ for $is\in N_B$ we associate
a number $d_{is}$ defined by
$$
d_{is}=\frac{\lambda^{(s)}_{1,i}}{\| e_{ii_s}\| }
(z_{i_s}-z_i)+\frac{\lambda^{(s)}_{2,i}}{\| e_{ii_{s+1}}\|
}(z_{i_{s+1}}- z_i)+ \frac{\lambda^{(s)}_{3,i}}{\| e_{ii_{s+2}}\|
}(z_{i_{s+2}}-z_i),
$$
which reflects the position of the data in the supporting set of
$B_{is}$.
The following
two lemmas are proved in~\citep{AEIV} for $p=2$ but they clearly hold for any $p$, $1<p<\infty$.
\begin{lem}\label{lemma1}
{Functions $B_{is},\ is\in N_B$,
are linearly independent in $E$.}
\end{lem}
\begin{lem}\label{lemma2}
{$F\in {\cal C}_p(E)\ \text{if and only if}\ \langle
F'',B_{is}\rangle=d_{is},\ is\in N_B$.}
\end{lem}

\section{Existence and uniqueness of the solution}

In the next theorem we prove that problem $(P_p)$ for $1< p<\infty$ always has a unique solution which we call {\it optimal curve network}.

\begin{thm}\label{th1}
{\it The extremal problem $(P_p)$ for $1<
p<\infty$ always has a unique solution . }
\end{thm}

\begin{proof} Let $1< p< \infty$. We recall that ${\cal C}_p(E)\not=\emptyset$. The set of real non-negative numbers
$\{\|F''\|_p\ |\ F\in {\cal C}_p(E)\}$ is bounded from below and therefore it has a greatest lower bound $d:=\inf_{F\in {\cal C}_p(E)}\|
F''\|_p$.
Let $\{F_\nu\}_{\nu=1}^{\infty}$, where $F_{\nu}:=\{f_{e,\nu}\}_{e\in E}\in
{\cal C}_p(E),$
be a minimizing sequence, i.~e. $\lim_{\nu\rightarrow\infty}\|F''_{\nu}\|_p = d$.
We denote
\begin{equation}\label{e220}
L_p(E):=\left\{G=\{g_e\}_{e\in E}\ :\ g_e\in L_p,\ e\in E\right\}.
\end{equation}

Next we prove that $\{F_\nu''\}_{\nu=1}^{\infty}$ is a fundamental sequence in $L_p(E)$, i.~e. $\lim_{\nu,\
\mu\rightarrow\infty}\|F_{\nu}''-F_{\mu}''\|_p=
0$.
 For this purpose we use the following Clarkson's inequalities (see~\citep{HS}, pp. 225, 227) which hold for $f,\, g\in L_p,\ 1<p<\infty,\
 1/p+1/q=1$,
\begin{eqnarray}
\|\frac{f+g}{2}\|_p^p+\|\frac{f-g}{2}\|_p^p&\leq &\frac{1}{2}\left(\|f\|_p^p+
\|g\|_p^p\right),\ \mbox{if}\ p\geq 2,\nonumber\\
&&\label{e210}\\
\|\frac{f+g}{2}\|_p^q+\|\frac{f-g}{2}\|_p^q&\leq &(\frac{1}{2}(\|f\|_p^p+
\|g\|_p^p))^{1/(p-1)},\ \mbox{if}\ 1<p<2.\nonumber
\end{eqnarray}
Since $F_{\nu},\,F_{\mu}\in {\cal C}_p(E)$ then
$(F_{\nu}+F_{\mu})/2\in {\cal C}_p(E)$. Hence
$\|(F_{\nu}''+F_{\mu}'')/2\|_p\geq d$.
Let $p\geq 2$. From the first inequality in (\ref{e210}) we obtain
\begin{eqnarray*}
\|\frac{F_{\nu}''-F_{\mu}''}{2}\|_p^p&\leq &\frac{1}{2}(\|F_{\nu}''\|_p^p+
\|F_{\mu}''\|_p^p)-\|\frac{F_{\nu}''+F_{\mu}''}{2}\|_p^p\\[1ex]
&\leq &\frac{1}{2}(\|F_{\nu}''\|_p^p+
\|F_{\mu}''\|_p^p)-d^p\rightarrow\frac{1}{2}(2d^p)-d^p=0\ \mbox{when}\ \nu ,\mu\rightarrow
\infty .
\end{eqnarray*}
For $1<p<2$ from the second inequality in (\ref{e210}) we obtain
\begin{eqnarray*}
\|\frac{F_{\nu}''-F_{\mu}''}{2}\|_p^q&\leq &(\frac{1}{2}(\|F_{\nu}''\|_p^p+
\|F_{\mu}''\|_p^p))^{1/(p-1)}-\|\frac{F_{\nu}''+F_{\mu}''}{2}\|_p^q\\[1ex]
&\leq &(\frac{1}{2}(\|F_{\nu}''\|_p^p+
\|F_{\mu}''\|_p^p))^{1/(p-1)}-d^q\\
&&\rightarrow (\frac{1}{2}(2d^p))^{1/(p-1)}-d^q=0\ \mbox{when}\ \nu
,\mu\rightarrow \infty .
\end{eqnarray*}

Therefore $\lim_{\nu,\
\mu\rightarrow\infty}\|F_{\nu}''-F_{\mu}''\|_p=
0$ for every $p,\ 1< p<\infty$, i.~e.
$\{F_{\nu}''\}_{\nu=1}^{\infty}$
is a fundamental sequence.
Since $L_p$ is a complete space then there exists curve network
$G=\{g_e\}_{e\in E}\in L_p(E)$ such that $g_e\in
L_p$ for every $e\in E$ and
\begin{equation}\label{e218}
\lim_{\nu\rightarrow\infty}\|F_{\nu}''-G\|_p=0.
\end{equation}
From (\ref{e218})
it follows that there exists a subsequence of
$\{F_{\nu}''\}_{\nu =1}^{\infty}$ that converges pointwise almost everywhere (a.e.)
to $G$. For simplicity we assume that
$\{F_{\nu}''\}_{\nu=1}^{\infty}$ is that subsequence, i.~e. $
\lim_{\nu\rightarrow\infty}f_{e,\nu}''(t)= g_e(t)\ \mbox{a.e. in}\
[0,\|e\|]. $
Moreover, from the continuity of the norm we have
$\lim_{\nu\rightarrow\infty}\|F_{\nu}''\|_p= \|G\|_p$, and hence
$\|G\|_p=d$.

Let $F=\{f_e\}_{e\in E}$ be the unique curve network that satisfies the interpolation conditions $F(V_i)=z_i,\ i=1,\dots ,n$ and  its
second derivative
$F''$ coincides a.e. with $G$.
To prove the existence of the solution to the problem, next we show that
$F=\{f_e\}_{e\in E}$ belongs to ${\cal C}_p(E)$. We have to show that for
every vertex $V_i$ there exists a tangent plane to the curve network $F$. First, we prove that
\begin{equation}\label{e213}
\lim_{\nu\rightarrow\infty}f_{e,\nu}'(t)=f_e'(t)
\mbox{ for every }\ t\in [0,\|e\|].
\end{equation}
Since $f_e''\in L_p$ then $f_e'\in AC$ and since $
\lim_{\nu\rightarrow\infty}f_{e,\nu}''(t)= f_e''(t)\ \mbox{a.e. in}\
[0,\|e\|] $, we have
 $$
\lim_{\nu\rightarrow\infty}\int_0^tf_{e,\nu}''(u)du=\int_0^tf_e''(u)du\
\mbox{for\ every}\ t\in [0,\|e\|],
$$
hence
\begin{equation}\label{e211}
\lim_{\nu\rightarrow
\infty}\left(f_{e,\nu}'(t)-f_{e,\nu}'(0)\right)=f_e'(t)-f_e'(0).
\end{equation}
From (\ref{e211}) after integration we obtain
$$
\lim_{\nu\rightarrow\infty}\left(f_{e,\nu}(t)-f_{e,\nu}(0)-tf_{e,\nu}'(0)\right
) = f_e(t)-f_e(0)-tf_e'(0),
$$
hence from the interpolation conditions we have
$$
\lim_{\nu\rightarrow\infty}\left(f_{e,\nu}(t)-tf_{e,\nu}'(0)\right)=
f_e(t)-tf_e'(0)\ \mbox{for\ every}\ t\in [0,\|e\|].
$$
In particular, for $t=\|e\|$ we obtain
$\lim_{\nu\rightarrow\infty}f_{e,\nu}'(0)=f_e'(0)$ and
(\ref{e213}) follows from (\ref{e211}).

Further on, if $\bf{a},\bf{b},\bf{c}$ are
vectors in $\mathbb{R}^3$, we denote by
$(\bf{a},\bf{b},\bf{c})$ their scalar triple product.

Functions $f_{e},\ e\in E$, defined by
(\ref{e0}) are parametric curves in $\mathbb{R}^3$ represented by
$(x=x_e(t),\ y=y_e(t),\ z=z_e(t)).$
Then vector ${\bf t_e}\in \mathbb{R}^3$ defined by
${\bf t_e}:=(\frac{x_j-x_i}{\| e_{ij}\| },
\frac{y_j-y_i}{\| e_{ij}\| },f'_{ij}(0))$ is a tangent vector to curve $f_e$ at point $V_i$. Let $V_i$ be a vertex in
$T$ of degree $m_i\geq 3$ (if $m_i=2$ then a tangent plane in $V_i$ always exists). Let $e_1$, $e_2$, and $e_3$ be three arbitrary edges incident to $V_i$. Since
$F_{\nu}\in{\cal C}_p(E)$ then $F_{\nu}$ has a tangent plane at $V_i$. A necessary and sufficient condition for the existence of such a
plane is
\begin{equation}\label{e240}
({\bf t_{e_1,\nu}, t_{e_2,\nu},
t_{e_3, \nu}})=0, \end{equation}
where ${\bf t_{e_1,\nu}},\ {\bf t_{e_2,
\nu}}$, and ${\bf t_{e_3,\nu}}$ are the three tangent vectors.
We take the limit in
 (\ref{e240}) for $\nu\rightarrow\infty$, use (\ref{e213}),
and obtain that the scalar triple product of the limit vectors ${\bf
t_{e_1}}, {\bf t_{e_2}}, {\bf t_{e_3}}$ is zero too.
The three edges $e_1$, $e_2$, and $e_3$ are arbitrarily chosen, hence the curve network $F$ has a tangent plane at point $V_i$, which has been
arbitrarily chosen too. Therefore
$F=\{f_e\}_{e\in E}$ belongs to ${\cal C}_p(E)$  and solves problem $(P_p)$.

It remains to prove uniqueness of the solution. Let $F_0$ and
$F_1$ be two solutions of $(P_p)$. Then
$\frac{1}{2}( F_0+F_1)\in {\cal C}_p(E)$.
From Minkowski's inequality it follows
\begin{equation}\label{e225}
\| \frac{1}{2}( F''_0+F''_1)\|_p\leq\frac{1}{2}\|F''_0\|_p+
\frac{1}{2}\|F''_1\|_p = \|F''_0\|_p =\|F''_1\|_p.
\end{equation}
Hence, in (\ref{e225}) we have equality which holds if and only if
$aF_0''=bF_1''\ $ a.e., where $a$ and $b$ are non-negative real numbers such that $a^2+b^2>0$. From the equality of the norms it
follows $a=b=1$, i.~e. $F''_0=F''_1\ $ a.e. which means $f_{0,e}''(t)=f_{1,e}''(t)$
a.e. in $[0,\|e\|]$ for every $e\in E$. Since $f_{0,e}(t)$ and $f_{1,e}(t)$
coincide at the endpoints of the edge $e$ then $f_{0,e}(t)=f_{1,e}(t)$ for every $t\in[0,\|e\|]$. Therefore
$F_0\equiv F_1$.   \qed
\end{proof}

\section{Characterization of the solution}

In this section we provide a full characterization of the solution $F^*$ to the extremal problem $(P_p)$
for $1<p<\infty$. Its existence and uniqueness have been already established in Theorem~\ref{th1}. Further, for simplicity we use the notation
$(x)^r_{\pm}\ :=\ |x|^r{\rm {sign}}(x),\ r\in \mathbb{R} ,\,x\in \mathbb{R}.$ Next we prove that
$(F^{*\prime\prime})_{\pm}^{p-1}$ can be represented as a linear combination of the basis curve networks defined by
(\ref{e5}). Finding of $F^*$ reduces to the unique solution of a system of equations.
The following theorem holds.
\begin{thm}\label{th2}
{\it Smooth interpolation curve network
$F^*=\{f^*_e\}_{e\in E}\in
{\cal C}_p(E)$ is a solution to problem $(P_p),\ 1<p<\infty$, if and only if
$$F^{*\prime\prime}=\left(\sum_{is\in N_B}\alpha_{is}B_{is}\right)
^{q-1}_{\pm}, $$ where $\alpha_{is}$ are real numbers and $1/p+1/q=1$.}
\end{thm}
\begin{proof}
Let us consider the set of interpolation curve networks
\begin{eqnarray*}
\Gamma (E):=\left\{F=\{f_e\}_{e\in E} : F(V_i)=z_i,\ i=1,\dots
,n,\ f_e'\in
AC,f''_e\in L_p,\ e\in E\right\}
\end{eqnarray*}
and the mapping
\be{e203}
\Gamma (E)\ni F=\{f_e\}_{e\in E}\ \mapsto\ \{g_e\}_{e\in E}=G\in
L_p(E),
\ee
where $G=\{g_e\}_{e\in E}$ is such that
$g_e=f_e''$, and the class $L_p(E)$ is defined by (\ref{e220}). If $F\in \Gamma(E)$
then obviously $F''=\{f_e''\}_{e\in
E}\in L_p(E)$. Now let $G=\{g_e\}_{e\in E}$ belong to $L_p(E)$.
We integrate twice the function $g_e,\ e=e_{ij}\in E$, use the two interpolation conditions
$F(V_i)=z_i$ and $F(V_j)=z_j$ at the end of the interval
$[0,\|e\|]$, and obtain curve network $F=\{f_e\}_{e\in E}$ such that
$F''=G$ and $F\in\Gamma(E)$. Therefore the mapping (\ref{e203}) is a bijection.
According to Lemma~\ref{lemma2},  (\ref{e203}) maps
the set ${\cal C}_p(E)\subset \Gamma (E)$ defined by (\ref{e31}) onto the following subset of
$L_p(E)$: $$\{G\ :\ G\in L_p(E),\ \int_E GB_{is}dt=d_{is},\ is\in N_B
\}.$$
Thus, problem $(P_p)$, $1<p<\infty$ is equivalent to the following problem
\begin{eqnarray}
&&\mbox{\it\ Find\ }{\tilde G}\in {\cal C}_p(E)\  \mbox{\it\
such\ that }\ \| {\tilde G}\|_p=\inf_{G\in L_p(E)}\|
G\|_p,\nonumber\\[-.2cm]
  ({\bf P}_p) \hspace*{1em} &&\label{e41}\\[-.2cm]
&&\mbox{\it under the conditions}\
\int _EG(t)B_{is}(t)dt=d_{is},\ is\in N_B.\nonumber
\end{eqnarray}

Using the Lagrange multipliers (see, e.g,~\cite{Sh}, pp. 113) we obtain that $G$
($F^{*\prime\prime}$, respectively) is a solution to problem $(P_p)$ for
$1<p<\infty$ if and only if there exist real numbers
$\lambda_{is},\ is\in N_B$ such that ${\tilde G}$ is a solution to the problem
\be{e167} \inf_{G\in L_p(E)}\left
(\int_E(|G(t)|^p-\sum_{is\in N_B}\lambda_{is}G(t)B_{is}(t))dt+\sum_{is\in N_B}\lambda_{is}d_{is}\right ). \ee Moreover, the partial derivative w.r.t. $G$ of the expression in the integral in (\ref{e167}) is zero for the extremal function ${\tilde G}$ (it follows from the Euler equation, see~\citep{Sh}, pp. 94). Hence, $$p|{\tilde
G}|^{p-1}{\rm {sign}}({\tilde G})-\sum_{is\in N_B}\lambda_{is}B_{is}=0.$$ From the last inequality it follows that the solution ${\tilde G}$ ($F^{*\prime\prime}$, respectively) has the form
\be{e22} F^{*\prime\prime}={\tilde G}=\left(\sum_{is\in N_B}\alpha_{is}B_{is} \right)_{\pm}^{q-1}, \ee where
$\alpha_{is}=\lambda_{is}/p,\ is\in N_B$, are real numbers and
$1/p+1/q=1$. Moreover, the representation (\ref{e22}) is unique. \qed
\end{proof}

As a consequence, we can formulate the following theorem.
\begin{thm}\label{th3}
{\it Curve network $F\in {\cal C}_p(E)$ solves problem $(P_p)$ for
$1<p<\infty$ if and only if $F^{\prime\prime}=\left(\sum_{is\in N_B}\alpha_{is}B_{is}\right )_{\pm}^{q-1}$.
The coefficients
$\alpha_{is}$ are the unique solution to the following system of equations
\be{e11}
\int_E\left(\sum_{is\in N_B}\alpha_{is}B_{is}
\right)_{\pm}^{q-1}B_{kl}dt=d_{kl},\ kl\in N_B.
\ee
}
\end{thm}

\begin{proof} From Theorem~\ref{th2} it follows that there exist real numbers
$\alpha_{is},\ is\in N_B$ such that the second derivative of the unique solution
$F$ to problem $(P_p)$ for $1<p<\infty$ is $F^{\prime\prime}=(\sum_{is\in N_B}\alpha_{is}B_{is}
)_{\pm}^{q-1}$. On the other hand, according to Lemma~\ref{lemma2}, $F$ is a smooth interpolation curve network if and only if $\langle F'',B_{kl}\rangle =d_{kl}$
for every $kl\in N_B$. Therefore numbers $\alpha_{is}$ are a solution to system (\ref{e11}).

The uniqueness of the solution to system (\ref{e11}) follows from the uniqueness of the optimal curve network and from the linear independence of the basic curve networks $B_{is},\ is\in N_B$, provided by Lemma~\ref{lemma1}. \qed
\end{proof}

\section{Examples and results}

To find the minimum $L_p$-norm networks for $1<p<\infty$ we have to solve system (\ref{e11}) which is nonlinear except in the case where $p=2$ when it is linear. We have adopted a Newton's algorithm \cite{V3} to solve this type of systems. We use Mathematica package to visualize the extremal curve networks. Below we present results of our experiments where solutions of
$(P_p)$ were computed and visualized for different $p$ on two small data sets.

\begin{example}\label{example1}
We consider data obtained from a regular triangular pyramid. We have $n=4$, $ V_1=(-1/2,-\sqrt{3}/6)$, $V_2=(1/2,-\sqrt{3}/6)$, $V_3=(0,\sqrt{3}/3)$, $V_4=(0,0)$, and $z_i=0$, $i=1,2,3$, $z_4=-1/2.$ The set of indices defining the edges of the corresponding triangulation is $N_B\,=\,\{ 12,23,31,41,42,43\}$.
We have $m_i=3$ for $i=1,\dots, 4$ and four basic curve networks $B_{is}$, $i=1,\dots, 4$, $s=1$, are defined.
The triangulation, the minimum ${L_p}$-norms network ${F_p}$, and the corresponding ${L_p}$-norms of the second derivatives ${\|F_p^{\prime\prime}\|_p}$ for ${p=2,\,3}$,\ {\rm and}\ $6$ are shown in Fig.~\ref{example} (left).
\end{example}

\begin{example}\label{example2}
We have $n=7$ and the data are $P_1=(-2,0,0)$, $P_2=(-1.6,0,-2)$, $P_3=(0,0,-3)$, $P_4=(1.6,0,-2.5)$, $P_5=(2,0,0)$, $P_6=(-0.5,2.3,-1.7)$, and $P_7=(0.5,-2,-1.9)$. The set of indices defining the edges of the corresponding triangulation $T_2$ is  $N_B\,=\,\{ 17,12,16,27,23,26,37,34,36,45,46,47,56,57\}$.
We have $m_1=m_5=3$, $m_2=m_3=m_4=4$, $m_6=m_7=5$, and hence, the number of the basic curve networks $B_{is}$ is fourteen.
The triangulation, the minimum ${L_p}$-norms network ${F_p}$, and the corresponding ${L_p}$-norms of the second derivatives ${\|F_p^{\prime\prime}\|_p}$ for ${p=2,\,3}$,\ {\rm and}\ $6$ are shown in Fig.~\ref{example} (right).
\end{example}

\begin{figure}[H]
\begin{minipage}[b]{3.in}
\centering
\includegraphics[width=.8\textwidth]{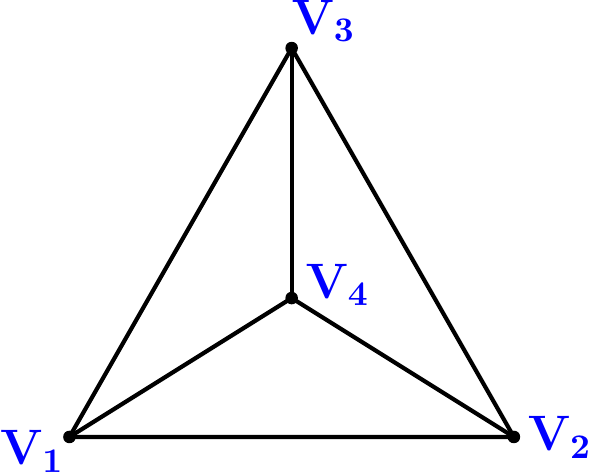}
\vspace*{.5cm}

\includegraphics[width=.9\textwidth]{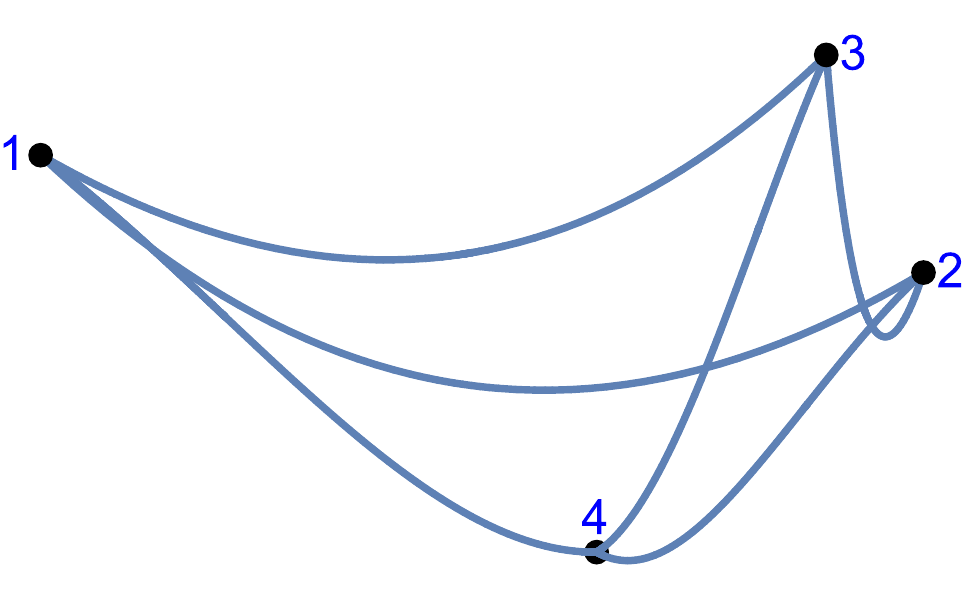}
\centering
{${p=2,\ \|F_2^{\prime\prime}\|_2=4.72119}$}
\includegraphics[width=.9\textwidth]{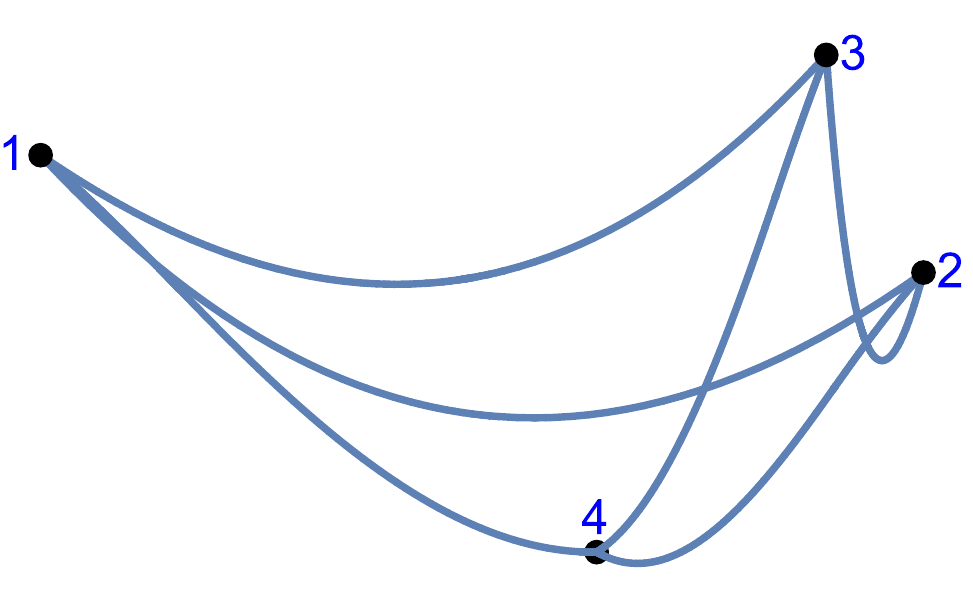}
\centering
{${p=3,\ \|F_3^{\prime\prime}\|_3=4.00185}$}
\includegraphics[width=.9\textwidth]{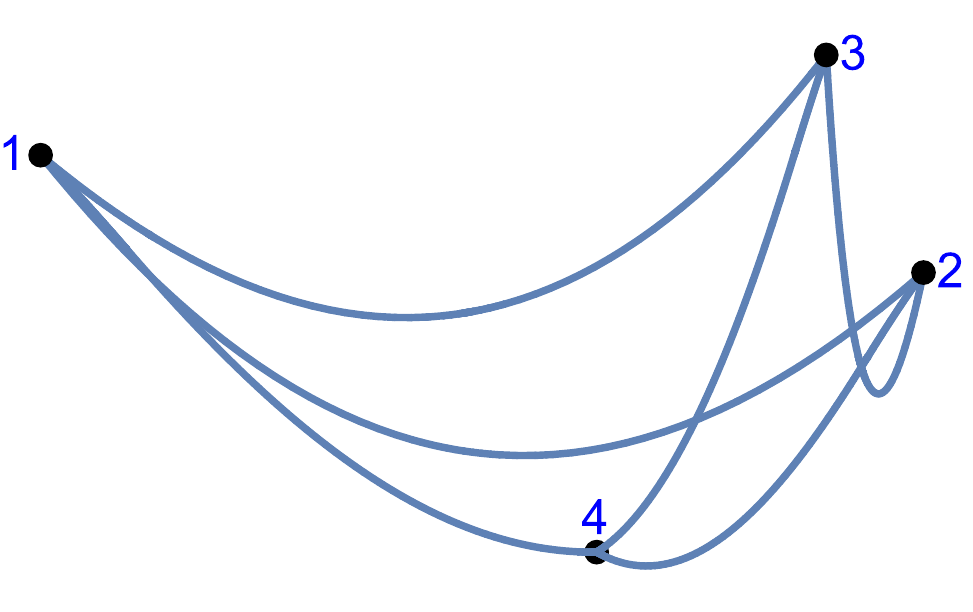}
\centering
{${p=6,\ \|F_6^{\prime\prime}\|_6=3.40846}$}
        \end{minipage}
        ~~~~~~~~~~~~~~~~~~~~~
               \begin{minipage}[b]{3.in}
               \hspace*{-.2cm}
            \centering
\includegraphics[width=.9\textwidth]{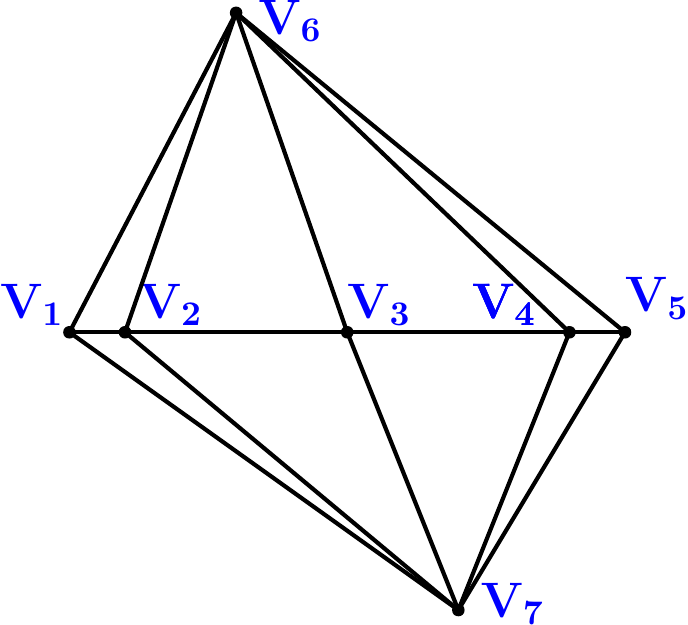}
\vspace*{.3cm}

\includegraphics[width=.92\textwidth]{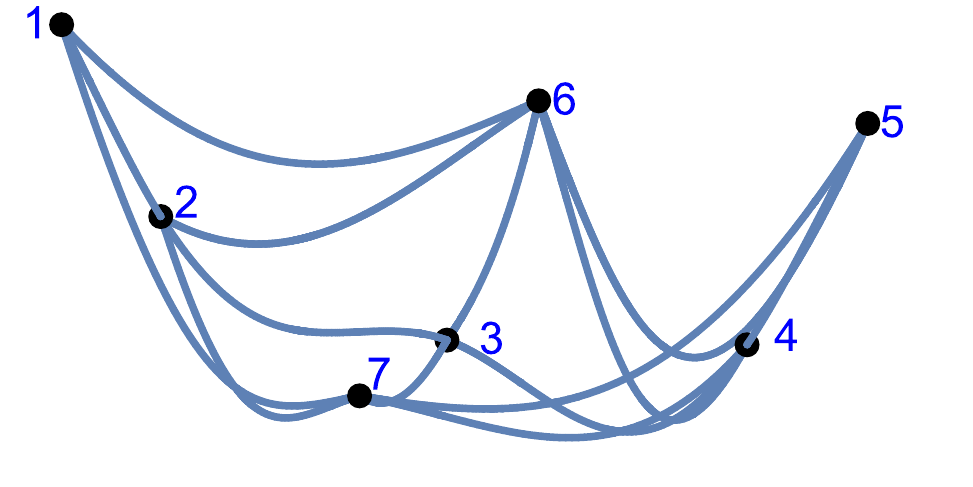}
\centering
{${p=2,\ \|F_2^{\prime\prime}\|_2=13.3007}$}
\vspace*{.4cm}

\includegraphics[width=.97\textwidth]{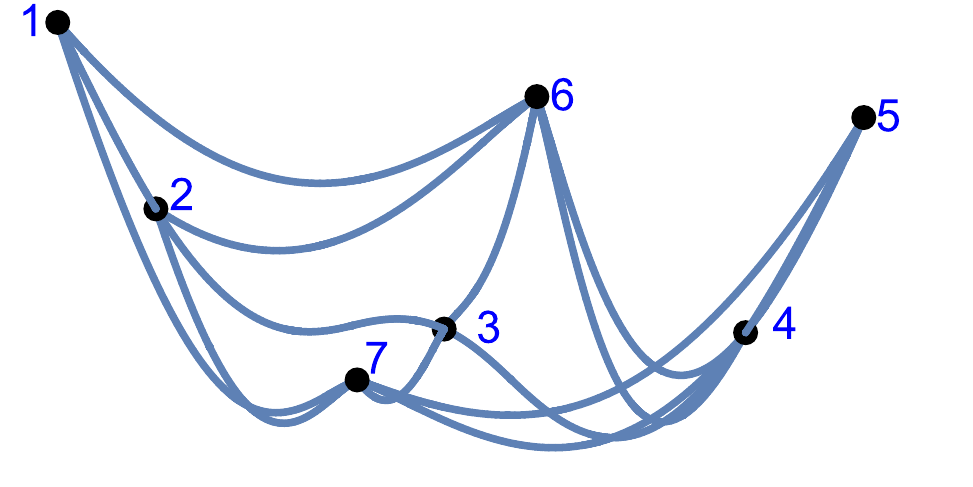}
\centering
{${p=3,\ \|F_3^{\prime\prime}\|_3=9.31125}$}
\vspace*{.9cm}

\includegraphics[width=.92\textwidth]{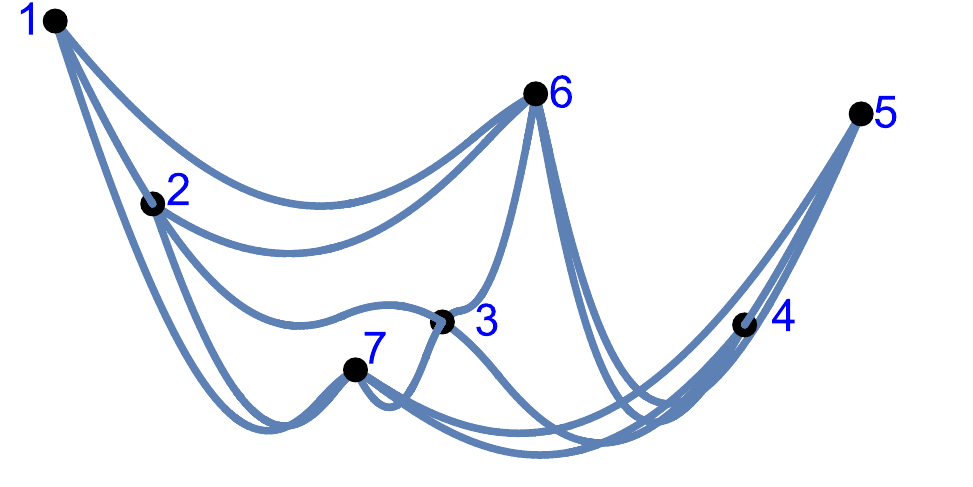}
\centering
{${p=6,\ \|F_6^{\prime\prime}\|_6=7.1822}$}
        \end{minipage}
        \caption{The triangulation, the minimum ${L_p}$-norm networks ${F_p}$, and the corresponding ${L_p}$-norms ${\|F_p^{\prime\prime}\|_p}$ for ${p=2,\,3}$,\ {\rm and}\ $6$ for Example~\ref{example1} (left), and Example~\ref{example2} (right).}\label{example}
        \end{figure}

\section{Conclusions and future work}
\label{sec:3}

In this paper we considered the extremal problem of interpolation of scattered data in $\mathbb{R}^3$ by smooth minimum $L_p$-norm networks for $1<p<\infty$. We proved the existence and the uniqueness of the solution for $1<p<\infty$ and provided its complete characterization.
We presented numerical experiments and gave examples to visualize and support the obtained results.

The case $p=\infty$ is not completely understood and needs to be studied further. First of all it is known that the solution in this case is not unique. Second, the approach based on Lagrange multipliers can not be applied directly to the case $p=\infty$. Another interesting question that arises is whether the sequence of solutions for $1<p<\infty$ converges as $p\rightarrow\infty$, and if yes, what is the limit?

\section*{Acknowledgments.} This work was supported by Sofia University Science Fund
Grant No. 80-10-145/2018, and by European Regional Development Fund and the Operational Program ``Science and Education for Smart Growth" under contract № BG05M2OP001-1.001-0004 (2018-2023).


\bibliographystyle{amsplain}
\bibliography{references}

\end{document}